\documentclass[12pt,twoside]{amsart}
\usepackage{amssymb}
\usepackage{verbatim}
\usepackage{amsmath}
\usepackage{bm}
\usepackage{a4wide}
\usepackage[latin1]{inputenc}
\usepackage[T1]{fontenc}
\usepackage{times}
\usepackage{amssymb,latexsym}
\usepackage{enumerate}
\usepackage[stable]{footmisc}
\usepackage{color}
\usepackage[colorlinks,linkcolor=black,citecolor=black, urlcolor=black]{hyperref}

\makeatletter
\newcommand{\sumprime}{\if@display\sideset{}{'}\sum%
            \else\sum'\fi}
\makeatother

\begin{document}

\numberwithin{equation}{section}

% define theorem environments
\newtheorem{theorem}{Theorem}[section]
\newtheorem{proposition}[theorem]{Proposition}
\newtheorem{conjecture}[theorem]{Conjecture}
\def\theconjecture{\unskip}
\newtheorem{corollary}[theorem]{Corollary}
\newtheorem{lemma}[theorem]{Lemma}
\newtheorem{observation}[theorem]{Observation}
\newtheorem{definition}{Definition}
\numberwithin{definition}{section} %\def\thedefinition{\unskip}
\newtheorem{remark}{Remark}
\def\theremark{\unskip}
\newtheorem{kl}{Key Lemma}
\def\thekl{\unskip}
\newtheorem{question}{Question}
\def\thequestion{\unskip}
\newtheorem{example}{Example}
\def\theexample{\unskip}
\newtheorem{problem}{Problem}

\thanks{The first author is supported by National Natural Sciense Foundation of China, No. 12271101. The second author is supported by China Postdoctoral Science Foundation, No. 2024M750487.}

\address [Bo-Yong Chen] {School of Mathematical Sciences, Fudan University, Shanghai, 200433, China}
\email{boychen@fudan.edu.cn}

\address [Yuanpu Xiong] {School of Mathematical Sciences, Fudan University, Shanghai, 200433, China}
\email{ypxiong@fudan.edu.cn}

\title{An estimate of the Bergman distance on  Riemann surfaces}
\author{Bo-Yong Chen and Yuanpu Xiong}

\date{}

\maketitle

\begin{center}
\emph{In memory of Marek Jarnicki}
\end{center}

\begin{abstract}
Let $M$ be a hyperbolic Riemann surface with the first eigenvalue $\lambda_1(M)>0$.  Let $\rho$ denote the distance from a fixed point  $x_0\in{M}$ and $r_x$ the injectivity radius at $x$.  We show that there  exists a numerical constant $c_0>0$ such that if
$
r_x\ge c_0 \lambda_1(M)^{-3/4} \rho(x)^{-1/2}
$
holds outside some compact set of $M$, then the Bergman distance verifies
$
d_B(x,x_0) \gtrsim  \log [1+\rho(x)].
$
\end{abstract}

\section{Introduction}
Completeness of the Bergman metric on complex manifolds,  initiated by the celebrated work of Kobayashi \cite{Kobayashi},  has been investigated by various authors in recent decades.  For more information on this matter,  we refer the reader to the comprehensive  book of Jarnicki-Pflug \cite{JPBook} or the survey article \cite{ChenSurvey} and the references therein.   There are also precise estimates of the Bergman distance on certain bounded hyperconvex domains in $\mathbb C^n$ (cf.  \cite{DO95},  \cite{Blocki05},  \cite{Chen17}) and K\"ahlerian Cartan-Hadamard manifolds (cf.  \cite{CZ02}).    

The goal of this note is to give an estimate of the Bergman distance in terms of hyperbolic geometry for  noncompact Riemann surfaces.  More precisely,  consider a noncompact {\it hyperbolic} Riemann surface $M$,  that is,  the universal covering of $M$ is the unit disc $\mathbb D$. The (Poincar\'e) hyperbolic metric on $\mathbb D$ descends to  the hyperbolic metric $ds^2_{\rm hyp}$ on $M$,  whose  Gauss curvature equals to $-1$. The geometry associated to  $ds^2_{\rm hyp}$ is called the hyperbolic geometry.

Following Kobayashi \cite{Kobayashi},  we define  ${\mathcal H}(M)$ to be the Hilbert space of holomorphic differentials $f$ on $M$ satisfying
\[
\|f\|^2:=\frac{i}2 \int_M f\wedge \bar{f}<\infty.
\]
Let $\{h_j\}_{j=1}^\infty$ be a complete orthonormal basis of ${\mathcal H}(M)$. The Bergman kernel $K_M$ of $M$ is given by
\[
K_M(x,y)=\sum_j h_j(x)\otimes \overline{h_j(y)}.
\]
In case $M$ is nonparabolic\footnote{We do not use "hyperbolic" as an antonym to "parabolic".},  i.e.,  it carries  the (negative) Green function,  it also carries the Bergman metric,  which is an invariant K\"ahler metric given by 
\[
ds_B^2:=\frac{\partial^2 \log K_M^\ast(z,z)}{\partial z \partial\bar{z}} dz\otimes d\bar{z},
\]
where 
$
K_M(z,z)=K_M^\ast(z,z)dz\otimes d\bar{z}
$
in local coordinates (compare \cite{CZ02}).  

To state our result,  let us first recall two fundamental concepts in (hyperbolic) geometry.   
Let $d$ be the distance function induced by $ds^2_{\rm hyp}$ and $B_r(x)$ the geodesic ball centred at $x$ with radius $r$. The\/ {\it injectivity radius}\/ at $x\in{M}$ is defined to be
\[
r_x:=\frac12 \inf_{\gamma\in\Gamma\setminus\{1\}}d(\widetilde{x},\gamma\widetilde{x}),
\]
which is independent of the choice of $\widetilde{x}\in\varpi^{-1}(x)$.  Here $\Gamma$ is a Fuchsian group so that $M=\mathbb D/\Gamma$ and $\varpi:\mathbb D\rightarrow M$ is the universal covering map.  Let $\Delta$ denote the (real) Laplace operator associated to $ds^2_{\rm hyp}$.  The bottom of the spectrum (or the first eigenvalue) of $-\Delta$ is given by
\[
\lambda_1(M)=\inf \left\{\frac{\int_M |\nabla\phi|^2 dV}{\int_M |\phi|^2 dV}:\phi\in C_0^\infty(M)\backslash \{0\}\right\}.
\]
It is a classical fact that $M$ is nonparabolic provided $\lambda_1(M)>0$ (cf.  \cite{Grigoryan}).

Our main result is given as follows.

\begin{theorem}\label{th:Main}
Let $M$ be a hyperbolic Riemann surface with $\lambda_1(M)>0$. Fix $x_0\in{M}$ and define $\rho(x):=d(x,x_0)$. There exists a numerical constant $c_0>0$ such that if
\begin{equation}\label{eq:injectivity_radius}
r_x\ge c_0 \lambda_1(M)^{-3/4} \rho(x)^{-1/2}
\end{equation}
holds outside some compact set of $M$, then the Bergman distance verifies
\begin{equation}\label{eq:Estimate}
d_B(x,x_0) \gtrsim  \log [1+\rho(x)],\ \ \  \forall\,x\in M.
\end{equation}
\end{theorem}

\begin{remark}
{\rm (1) The punctured disc $\mathbb D^\ast$ satisfies $\lambda_1(\mathbb D^\ast)>0$ but is not Bergman complete. This shows that the conclusion fails if $r_x$ decays rapidly at infinity.}
   
{\rm (2) If $\lambda_1(M)>0$ and  $\inf_{x\in{M}}r_x>0$,  then the Bergman metric is quasi-isometric to the hyperbolic metric (cf. \cite{ChenEssay}).}
\end{remark}
   
It is known that 
$$
r_x\gtrsim|B_1(x)|\ge e^{-C\rho(x)}
$$
 holds on any hyperbolic Riemann surface  (cf. \cite{CGT,SchoenYau}).  Therefore,   we would like  to ask the following
  
\begin{problem}
Let $M$ be a hyperbolic Riemann surface with $\lambda_1(M)>0$.  Is it possible to find $\varepsilon>0$ such that  $r_x\gtrsim e^{-\varepsilon \rho(x)}$ implies Bergman completeness of $M$?
\end{problem}
 
\section{Preliminaries}

\subsection{Several conditions equivalent to $\lambda_1(M)>0$}
    
By the uniformization theorem, we may write $M=\mathbb{D}/\Gamma$ for suitable  Fuchsian group $\Gamma$.  It follows from  a classical theorem of Myrberg (cf. \cite{Tsuji}) that $M$ is nonparabolic if and only if 
\[
\sum_{\gamma\in \Gamma} \,(1-|\gamma(0)|)<\infty.
\]
The \/ {\it critical exponent}\/  of Poincar\'e series is given by
\[
\delta(M):=\inf\left\{s\ge 0: \sum_{\gamma\in \Gamma}\, (1-|\gamma(0)|)^s<\infty\right\}.
\]
It is known that $\delta(M)\le 1$ (cf. \cite{Tsuji}).

Recall that the\/ {\it isoperimetric constant}\/  of $M$ is defined  as
\[
I(M)=\inf_{\Omega\in \mathcal F} \frac{|\partial \Omega|}{|\Omega|},
\]
where the supremum is taken over all precompact domains with  smooth boundary in $M$.  Here
$| \partial \Omega|$ and $|\Omega|$ denote the (hyperbolic) volume of $\partial \Omega$ and $|\Omega|$ respectively.
    
There are some striking relationships between these quantities and $\lambda_1(M)$: 
    
(1) Cheeger's inequality \cite{Cheeger}: $\lambda_1(M)\ge I(M)^2/4$; 
    
(2) Buser's inequality \cite{Buser}: $\lambda_1(M)\le C_0 I(M)$ for some numerical constant $C_0>0$; 
    
(3) Elstrodt-Patterson-Sullivan theorem \cite{Sullivan}:
\[
\lambda_1(M) = 
\begin{cases}
1/4\ \ \ &\text{if}\ 0\le \delta(M)\le 1/2,\\
\delta(M)(1-\delta(M)) &\text{if}\ 1/2<\delta(M)\le 1.
\end{cases}
\]
In particular, we have
\[
\lambda_1(M)>0 \iff I(M)>0 \iff \delta(M)<1\Rightarrow M\ \text{is nonparabolic}.
\]
This shows that the class of Riemann surfaces with $\lambda_1(M)>0$ is quite large.  

\subsection{Capacity}
Given a compact set $E$ in $M$,  define the capacity ${\rm cap}(E)$ of $E$ by
\[
{\rm cap}(E)=\inf\int_M  |\nabla \phi|^2 dV
\]
where the infimum is taken over all $\phi\in C^\infty_0(M)$ such that $0\le \phi\le 1$ and $\phi|_E=1$.  For any $c< \lambda_1(M)$,  we have 
\[
c \int_M|\phi|^2dV\leq\int_M|\nabla\phi|^2dV,\ \ \ \forall\,\phi\in{C^\infty_0(M)}.
\]
It follows that for any $\phi\in C_0^\infty(M)$ with $0\le \phi\le 1$ and $\phi|_E=1$,  
\[
\int_M|\nabla\phi|^2dV\geq c |E|,
\]
so that ${\rm cap}(E)\ge c|E|$.  Letting $c\rightarrow \lambda_1(M)-$,  we obtain 
\begin{equation}\label{eq:CapVol}
{\rm cap}(E)\geq\lambda_1(M)|E|.
\end{equation}

Let $g_M(x,y)$ denote the (negative) Green function on $M$, i.e., given any local coordinate $z$ near $y$ with $z(y)=0$, $g_M(\cdot,y)$ is the supremum of all negative subharmonic functions $u$ on $M$ with
\[
u(x)=\log|z(x)|+O(1)
\]
as $x\rightarrow{y}$. It follows from Proposition 4.1 in \cite{Grigoryan} that for every open set $\Omega\subset\subset M$
\begin{equation}\label{eq:GreenVsCapacity}
\inf_{\partial\Omega} [-g_M(\cdot,y)]\le \frac{2\pi} {{\rm cap}(\overline{\Omega})}\le  \sup_{\partial\Omega} [-g_M(\cdot,y)],\ \ \ {\forall\,} y\in \Omega
\end{equation}
(see also \cite{ChenJLMS},  Lemma 3.3).

\subsection{A Harnack inequality}
Let us write $M=\mathbb D/\Gamma$ for suitable Fuchsian group $\Gamma$ and let $\varpi:\mathbb{D}\rightarrow{M}$ be the universal covering map.  Given $x\in{M}$ and $\widetilde{x}\in\varpi^{-1}(x)$,  consider the  fundamental domain
\[
D:=\left\{z\in\mathbb{D}:d(z,\widetilde{x})<\inf_{\gamma\in\Gamma\setminus\{1\}}d(z,\gamma\widetilde{x})\right\}\ni\widetilde{x},
\]
such that $\varpi|_D$ is injective and $M\setminus\varpi(D)$ is of zero measure. It follows that $B_{r_x}(\widetilde{x})\subset{D}$,  so that $\varpi: B_{r_x}(\widetilde{x})\rightarrow B_{r_x}(x)$ is a homeomorphism, where $r_x$ is the injectivity radius at $x$.

Let $g_M(\cdot,x)$ be the Green function of $M$ with a logarithmic pole at $x$ and set 
$$
u:=\varpi^*g_M(\cdot,x)=g_M(\varpi(\cdot),x).
$$
Clearly,  $u$ is harmonic on $\mathbb{D}\setminus\varpi^{-1}(x)$.  In particular, given any $\widetilde{x}\in\varpi^{-1}(x)$, $u$ is harmonic on $B_{2r_x}(\widetilde{x})\setminus\{\widetilde{x}\}$. Set
\[
\widehat{r}_x:=\min\{r_x,1\}.
\]
We have the following Harnack inequality for $u$.

\begin{proposition}\label{prop:Harnack}
There exists a numerical constant $C_0$ such that
\[
\sup_{\partial B_{\widehat{r}_x}(\widetilde{x})}(-u)\leq C_0\inf_{\partial B_{\widehat{r}_x}(\widetilde{x})}(-u).
\]
\end{proposition}

The idea is to find a chain of discs covering $\partial B_{\widehat{r}_x}(\widetilde{x})$, so that the classical Harnack inequality applies. More precisely, let us first verify the following

\begin{lemma}\label{lm:covering}
There exists a numerical integer $N_0$ such that for any $0<r\leq1$ and $z\in\mathbb{D}$, one can find $z_1,\cdots,z_{N_0}\in\partial{B_r(z)}$ with
\begin{equation}\label{eq:Harnack_disc}
\partial{ B_r(z)}\subset\bigcup^{N_0}_{j=1}B_{r/2}(z_j).
\end{equation}
\end{lemma}

\begin{proof}
Since the group of M\"{o}bius transformations of $\mathbb{D}$ acts transitively, we may assume that $z=0$. Denote $ds^2_{\mathrm{eucl}}$ the Euclidean metric of $\mathbb{C}$. It is easy to see that there exists numerical constant $C>1$ such that
\[
C^{-1}ds^2_{\mathrm{eucl}}\leq ds^2_{\mathrm{hyp}}\leq Cds^2_{\mathrm{eucl}}
\]
on $B_{1}(0)$. Thus $B_r(0)\subset \Delta(0,Cr)$ and $\Delta(\zeta,C^{-1}r/2)\subset B_{r/2}(\zeta)$ for any $\zeta\in\partial{B_r(0)}$ and $0<r\le 1$, where $\Delta(\zeta,s)$ denotes a Euclidean disc centred at $\zeta$ with radius $s$. It follows that every hyperbolic disc ${B}_{r/2}(\zeta)$ covers an arc of $\partial{B}_r(0)$ with a central angle larger than $2\theta$, where
\[
\theta=\arccos\frac{C^2r^2+C^2r^2-C^{-2}r^2/4}{2(Cr)^2}=\arccos\left(1-\frac{1}{8C^4}\right),
\]
in view of the law of cosine.  Thus $\partial{B}_r(0)$ can be covered by $N_0:=[\pi/\theta]+1$ hyperbolic discs ${B}_{r/2}(z_j)$, where $z_j\in\partial{B}_r(0)$ and $j=1,2,\cdots,N_0$.
\end{proof}

\begin{proof}[Proof of Proposition \ref{prop:Harnack}]
By Lemma \ref{lm:covering}, we have
\[
\partial{B}_{\widehat{r}_x}(\widetilde{x})\subset\bigcup^{N_0}_{j=1}B_{\widehat{r}_x/2}(z_j),
\]
where $z_1,\cdots,z_{N_0}\in\partial{B}_{\widehat{r}_x}(\widetilde{x})$. The Harnack inequality gives
\[
\frac{1}{3}\leq\frac{u(z)}{u(w)}\leq3,\ \ \ \forall\,z,w\in {B}_{\widehat{r}_x/2}(z_j),
\]
so that
\[
\frac{1}{3^{N_0}}\leq\frac{u(z)}{u(w)}\leq3^{N_0},\ \ \ \forall\,z,w\in \partial{B}_{\widehat{r}_x}(\widetilde{x}).
\]
Thus Proposition \ref{prop:Harnack} holds with $C_0=3^{N_0}$.
\end{proof}

To simplify notations,  let us write
\[
B_x:=B_{\widehat{r}_x}(x)=\begin{cases}
B_{r_x}(x),\ \ \ &x\leq 1,\\
B_1(x),\ \ \ &x>1.
\end{cases}
\]
Proposition \ref{prop:Harnack} and \eqref{eq:GreenVsCapacity} imply
\[
\sup_{\partial{B_x}}(-g_M(\cdot,x))\leq{C_0}\mathrm{cap}(\overline{B}_x)^{-1},
\]
for suitable numerical constant $C_0>0$ (different from the one in Proposition \ref{prop:Harnack}),  so that
\[
\{g_M(\cdot,x)\leq-{C_0}\mathrm{cap}(\overline{B}_x)^{-1}\}\subset\overline{B}_x,\ \ \ \forall\,x\in{M}.
\]
This combined with \eqref{eq:CapVol} yields
\begin{equation}\label{eq:including_Green}
\{g_M(\cdot,x)\leq-{C_0}\lambda_1(M)^{-1}|B_x|^{-1}\}\subset\overline{B}_x,\ \ \ \forall\,x\in{M}.
\end{equation}

\section{An estimate for the $L^2$-minimal solution of the $\bar\partial$-equation}
Since the complex Laplace operator is given by $\Box = \frac14 \Delta$, it follows that 
\begin{equation}\label{eq:eigenvalue}
\lambda_1(M)=4 \inf \left\{\frac{\int_M |\partial \phi|^2}{\int_M |\phi|^2}:\phi \in C_0^\infty(M)\backslash \{0\}\right\}.
\end{equation}
Here and in what  follows in this section we  denote by $C^\infty_0(M)$ the set of complex-valued smooth functions with compact support in $M$.  To see \eqref{eq:eigenvalue},  simply note that for every $\phi\in C_0^\infty(M)$, 
\[
\int_M |\bar{\partial}\phi|^2 dV=\frac{i}2\int_M \partial \bar{\phi}\wedge \bar{\partial} \phi  =- \frac{i}2\int_M \bar{\phi} \partial\bar{\partial} \phi=\frac{i}2\int_M \partial {\phi}\wedge \bar{\partial} \bar{\phi}=\int_M |\partial \phi|^2 dV,
\]
so that
\[
 \int_M |\nabla\phi|^2 dV = 2\int_M |\partial \phi|^2dV+2 \int_M |\bar{\partial} \phi|^2 dV=4\int_M |\partial \phi|^2 dV.
\]

Let $\varphi$ be a continuous real-valued function  on $M$. Let $D_{(p,q)}(M)$ be the set of smooth $(p,q)$ forms with compact support in $M$ and let   $L^2_{(p,q)}(M,\varphi)$  be the completion of $D_{(p,q)}(M)$ with respect to the following norm 
\[
\|f\|_\varphi^2:= \int_M |f|^2 e^{-\varphi} dV.
\]
Here $|f|$ and $dV$ denote the point-wise length and the volume associated the hyperbolic metric $ds^2_{\rm hyp}$. It is important to remak that if $f$ is a $(1,0)$  form then 
\[
\|f\|_\varphi^2 =\frac{i}2 \int_M f\wedge \bar{f} e^{-\varphi},
\]
which is essentially independent of $ds^2_{\rm hyp}$.

\begin{lemma}\label{lm:L2}
Let $\varepsilon,\tau$ be positive numbers satisfying
\[
(1+\varepsilon)\tau<\lambda_1(M).
\]
Let $\varphi$ be a Lipschitz continuous real-valued function on $M$ such that
\[
|\partial \varphi|^2 \le \tau\ \ \ a.e.
\]
For any $v\in L^2_{(1,1)}(M,\varphi)$, there exists a solution of\/ $\bar{\partial}u=v$ such that
\[
\|u\|_\varphi\le \sqrt{C_{\varepsilon,\tau}}\,\|v\|_\varphi
\]
where
\[
C_{\varepsilon,\tau}=\frac{4(1+\varepsilon^{-1})}{\lambda_1(M)-(1+\varepsilon)\tau}.
\]
\end{lemma}

\begin{proof}
With $ds^2_{\rm hyp}=\mu(z) dz\otimes d\bar{z}$ we define two inner products 
\begin{eqnarray*}
(f_1,f_2)
&=& \int_M \phi_1 \bar{\phi}_2 \mu^{-1}dV_z\\
(f_1,f_2)_\varphi
&=& \int_M \phi_1 \bar{\phi}_2 \mu^{-1} e^{-\varphi} dV_z
\end{eqnarray*}
where $f_1=\phi_1 dz\wedge d\bar{z},\,f_2=\phi_2 dz\wedge d\bar{z}\in D_{(1,1)}(M)$, and $dV_z=\frac{\sqrt{-1}}2 dz\wedge d\bar{z}$. Let $\bar{\partial}^\ast$ and $\bar{\partial}^\ast_\varphi$ be the formal adjoint of $\bar{\partial}$ associated to $(\cdot,\cdot)$ and $(\cdot,\cdot)_\varphi$ respectively. For all $u=\psi dz\in D_{(1,0)}(M)$ and $f=\phi dz\wedge d\bar{z}\in D_{(1,1)}(M)$, we have
\[
(f,\bar{\partial}u)=-\int_M \phi \frac{\partial \bar{\psi}}{\partial z} \mu^{-1} dV_z=\int_M \frac{\partial}{\partial z}(\mu^{-1}\phi)\bar{\psi} dV_z,
\]
so that
\[
\bar{\partial}^\ast f=\frac{\partial}{\partial z}(\mu^{-1}\phi) dz.
\]
Since $\tilde{f}:=\mu^{-1}\phi\in C_0^\infty(M)$, it follows that $\bar{\partial}^\ast f=\partial \tilde{f}$ and (\ref{eq:eigenvalue}) implies 
\begin{equation}\label{eq:PoincareIneq}
\int_M |f|^2 dV\le \frac4{\lambda_1(M)}\int_M |\bar{\partial}^\ast f|^2 dV.
\end{equation}
Analogously,  we have
\[
(f,\bar{\partial}u)_\varphi=-\int_M \phi \frac{\partial \bar{\psi}}{\partial z} \mu^{-1} e^{-\varphi} dV_z=\int_M \frac{\partial}{\partial z}(\mu^{-1}\phi e^{-\varphi})\bar{\psi} dV_z,
\]
so that
\[
\bar{\partial}^\ast_\varphi f=\left(\frac{\partial}{\partial z}(\mu^{-1}\phi)-\mu^{-1}\phi\frac{\partial\varphi}{\partial z}\right) dz=:\bar{\partial}^\ast f+\bar{\partial} \varphi \lrcorner\, f.
\]
By (\ref{eq:PoincareIneq}), we have
\begin{eqnarray*}
\int_M |f|^2 e^{-\varphi} dV
&=& \int_M |fe^{-\varphi/2}|^2 dV\\
&\leq& \frac4{\lambda_1(M)}\int_M |\bar{\partial}^\ast (fe^{-\varphi/2})|^2dV\\
&=& \frac4{\lambda_1(M)}\int_M \left|\bar{\partial}^\ast f+\frac12\bar{\partial} \varphi \lrcorner\,f\right|^2e^{-\varphi}dV\\
&=& \frac4{\lambda_1(M)}\int_M \left|\bar{\partial}^\ast_\varphi f-\frac12\bar{\partial} \varphi\lrcorner\,f\right|^2e^{-\varphi}dV\\
&\leq& \frac4{\lambda_1(M)}\left\{(1+\varepsilon^{-1})\|\bar{\partial}^\ast_\varphi f\|_\varphi^2+(1+\varepsilon)\frac{\tau}4 \|f\|_\varphi^2\right\},
\end{eqnarray*}
so that
\[
\|f\|^2_\varphi \le C_{\varepsilon,\tau}\|\bar{\partial}^\ast_\varphi f\|_\varphi^2.
\]
The remaining argument is standard. Given $v\in L^2_{(1,1)}(M,\varphi)$,  the linear functional
\[
{\rm Range\,}\bar{\partial}^\ast_\varphi \rightarrow {\mathbb C},\ \ \ \ \ \bar{\partial}^\ast_\varphi f \mapsto (f,v)_\varphi
\]
is bounded by $\sqrt{C_{\varepsilon,\tau}}\|v\|_\varphi$. Thus by Hahn-Banach's theorem and the Riesz representation theorem, there is a unique $u\in L^2_{(1,0)}(M,\varphi)$ such that
\[
(\bar{\partial}^\ast_\varphi f,u)_\varphi=(f,v)_\varphi
\]
for all $f\in \mathcal{D}_{(1,1)}(M)$, i.e., $\bar{\partial}u=v$ holds in the sense of distributions, such that
\[
\int_M |u|^2e^{-\varphi} dV \le C_{\varepsilon,\tau}\int_M |v|^2e^{-\varphi} dV. \qedhere
\]
\end{proof}

\begin{proposition}\label{prop:DF}
Let $\varphi$ be a  Lipschitz continuous real-valued function on $M$ which satisfies
\[
|\partial \varphi|^2\le \tau<\lambda_1(M)/9\ \ \ a.e.
\]
Let $v\in L^2_{(1,1)}(M)$ and let $u_0$ be the $L^2$\/ {\rm minimal\/} solution of the equation $\bar{\partial}u=v$. Then
\[
\|u_0\|_{-\varphi}\le {\rm const}_\tau\|v\|_{-\varphi}.
\]
\end{proposition}

   \begin{proof}
   We will employ a trick from \cite{BC} to get the desired estimate. Let $M$ be exhausted by a sequence of precompact open subsets $\{\Omega_n\}$ with smooth boundaries. Let $\lambda_1(\Omega_n)$ be the infimum of the spectrum of $-\Delta$ on $\Omega_n$. It is easy to see that Lemma \ref{lm:L2} remains valid if $M$ is replaced by $\Omega_n$. Let $u_n$ be the $L^2$ minimal solution of $\bar{\partial} u=v$ on $\Omega_n$.
    Since $\varphi$ is bounded on $\Omega_n$ and $u_n\bot {\rm Ker\,}\bar{\partial}$ in $L^2_{(1,0)}(\Omega_n)$, we conclude that $u_ne^{\varphi}\bot {\rm Ker\,}\bar{\partial}$ in $L^2_{(1,0)}(\Omega_n,\varphi)$, so that  Lemma \ref{lm:L2} yields
\begin{eqnarray*}
\int_{\Omega_n} |u_n|^2e^{\varphi} dV
&=& \int_{\Omega_n} |u_ne^{\varphi}|^2e^{-\varphi} dV\\
&\leq&  C_{\varepsilon,\tau,n} \int_{\Omega_n} |\bar{\partial}(u_ne^{\varphi})|^2e^{-\varphi} dV\\
&\leq& C_{\varepsilon,\tau,n}\left\{(1+\delta^{-1})\int_{\Omega_n} |v|^2 e^{\varphi} dV+(1+\delta)\tau \int_{\Omega_n} |u_n|^2 e^{\varphi} dV \right\}
\end{eqnarray*}
for all $\delta>0$, provided
\[
(1+\varepsilon)\tau<\lambda_1(M)\le \lambda_1(\Omega_n).
\]
Here 
\[
C_{\varepsilon,\tau,n}=\frac{4(1+\varepsilon^{-1})}{\lambda_1(\Omega_n)-(1+\varepsilon)\tau}\le C_{\varepsilon,\tau}.
\]
Thus
\[
\int_{\Omega_n} |u_n|^2e^{\varphi} dV \le \frac{C_{\varepsilon,\tau}(1+\delta^{-1})}{1-(1+\delta)\tau C_{\varepsilon,\tau}}\int_{M} |v|^2e^{\varphi}dV
\]
    provided
    $$
    (1+\delta)\tau<C_{\varepsilon,\tau}^{-1}.
    $$
    We may take a subsequence of $\{u_n\}$ which converge weakly to $u_0$ such that
    $$
    \|u_0\|^2_{-\varphi} \le \frac{C_{\varepsilon,\tau}(1+\delta^{-1})}{1-(1+\delta)\tau C_{\varepsilon,\tau}}\,\|v\|_{-\varphi}^2.
    $$
    We look for the best $\tau$ which satisfies
    $$
    (1+\varepsilon)\tau<\lambda_1(M)\ \ \ {\rm and\ \ \ } \tau<C_{\varepsilon,\tau}^{-1}.
    $$
Note that $\tau<C_{\varepsilon,\tau}^{-1}$ if and only if
\[
\tau<\frac{\varepsilon}{(1+\varepsilon)(4+\varepsilon)}\lambda_1(M),
\]
whereas the function $\varepsilon/(1+\varepsilon)(4+\varepsilon)$ attains its maximum $1/9$ at $\varepsilon=2$. In other words, $\tau<\lambda_1(M)/9$ is the best possible.
\end{proof}

\section{Upper bounds for the off-diagonal Bergman kernel}
Let us write $ds^2_{\rm hyp}=\mu(z) dz \otimes d\bar{z}$ in local coordinates. Define 
\[
|K_M(x,y)|^2:=\frac{|K_M^\ast(x,y)|^2}{\mu(x)\mu(y)}
\]
and 
\[
\mathcal B_M(x,y):= \frac{|K_M(x,y)|^2}{|K_M(x,x)||K_M(y,y)|}=\frac{|K_M^\ast(x,y)|^2}{K_M^\ast (x,x) K_M^\ast(y,y)}.
\]
Let $d_B$ be the Bergman distance. By Kobayashi's theory \cite{Kobayashi}, we have the following fundamental inequality
\begin{equation}\label{eq:KI}
d_B(x,y)\ge \sqrt{1-\mathcal B_M(x,y)}.
\end{equation}
The goal of this section is to give an upper estimate for $\mathcal B_M(x,y)$ when $x\neq y$.

Let $\{h_j\}$ be a complete orthonormal basis of ${\mathcal H}$. Given $y\in{M}$ and a local coordinate $w$ near $y$, define a holomorphic differential by
\[
f_y(\cdot)=\sum_j \overline{h_j^\ast(y)}\,h_j(\cdot),
\]
where $h_j=h_j^\ast dw$. It follows that
\[
K_M(\cdot,y)=f_y(\cdot)\otimes d\bar{w}.
\]
For a function $\eta:M\rightarrow\mathbb(1,+\infty)$, we set
\[
A_\eta(x):=\{g_M(\cdot,x)\leq-\eta(x)\},\ \ \ \forall\,x\in{M}.
\]

\begin{lemma}\label{lm:upper_estimate_B}
If $A_\eta(x)\cap A_\eta(y)=\varnothing$, then there exists a numerical constant $C_1>0$ such that
\begin{equation}\label{eq:Bergman_upper_1}
\mathcal{B}_M(x,y)
\leq C_1e^{4\eta(x)}\frac{\int_{A_\eta(x)}|f_y|^2dV}{K_M^\ast(y,y)}.
\end{equation}
\end{lemma}

\begin{proof}
Let $\kappa:\mathbb R\rightarrow [0,1]$ be a cut-off function such that $\kappa|_{(-\infty,-\log2]}=1$ and $\kappa|_{[0,+\infty)}=0$. Since $g_M(\cdot,x)$ is a  negative harmonic function on $M\backslash \{x\}$ which satisfies
\[
-i\partial\bar{\partial} \log (-g_M(\cdot,x))\ge i\partial \log (-g_M(\cdot,x))\wedge \bar{\partial} \log (-g_M(\cdot,x)),
\]
we infer from the Donnelly-Fefferman estimate (cf. \cite{DF83},  see also \cite{BC}) that there exists a solution of the equation
\[
\bar{\partial} u=f_y \bar{\partial}\kappa(-\log(-g_M(\cdot,x))+\log \eta(x))
\]
such that
\begin{eqnarray*}
&& \int_M |u|^2 e^{-2g_M(\cdot,x)} dV\\
&\leq & C_1 \int_M |f_y|^2 |\bar{\partial}\kappa(-\log(-g_M(\cdot,x))+\log \eta(x))|^2_{-i\partial\bar{\partial}\log(-g_M(\cdot,x))} e^{-2g_M(\cdot,x)}dV\\
&\leq& C_1 e^{4\eta(x)} \int_{A_\eta(x)} |f_y|^2dV
\end{eqnarray*}
for some generic numerical constant $C_1>0$. Set
\[
F:=f_y \kappa(-\log(-g_M(\cdot,x))+\log \eta(x))-u.
\]
Clearly,  we have $F\in \mathcal H$,  and since $g_M(\cdot,x)$ has a logarithmic pole at $x$, we have $u(x)=0$ so that $F(x)=f_y(x)$; moreover, 
\begin{eqnarray*}
\int_M|F|^2dV
&\leq& 2\int_{A_\eta(x)} |f_y|^2 dV+2\int_M |u|^2 dV\\
&\leq& \left(2+2C_1e^{4\eta(x)}\right) \int_{A_\eta(x)} |f_y|^2 dV
\end{eqnarray*}
since $g_M(\cdot,x)<0$. Thus we get
\[
|K_M(x,x)|\ge \frac{|F(x)|^2}{\|F\|^2}\ge \left(2+2C_1e^{4\eta(x)}\right)^{-1} \frac{|f_y(x)|^2}{\int_{A_\eta(x)} |f_y|^2 dV},
\]
so that 
\[
{\mathcal B}_M(x,y)\le \left(2+2C_1e^{4\eta(x)}\right) K_M^\ast(y,y)^{-1}\int_{A_\eta(x)} |f_y|^2 dV,
\]
from which the assertion immediately follows.
\end{proof}

From now on,  let us fix
\[
\eta(x):={C_0}\lambda_1(M)^{-1}|B_x|^{-1},
\]
so that $A_\eta(x)\subset B_x$ in view of \eqref{eq:including_Green}.  Set $2B_x=B_{2\widehat{r}_x}(x)$. We also need the following

\begin{lemma}\label{lm:Bergman_disc}
If $d(x,y)\geq 2(\widehat{r}_x+\widehat{r}_y)$, i.e., $2B_x\cap 2B_y\neq\varnothing$, then for every $0<\tau<\lambda_1(M)/9$ there exists a constant $C=C_\tau>0$ such that
\[
\int_{B_x} |f_y|^2 dV  \le C  K_{B_y}^\ast(y,y)^{1/2} K_M^\ast(y,y)^{1/2} \widehat{r}_x^{\,-1}e^{\sqrt{\tau} [\rho(y)-\rho(x)]}.
\]
\end{lemma}

\begin{proof}
Let $\chi:\mathbb R\rightarrow [0,1]$ be a cut-off function such that $\chi|_{(-\infty,1]}=1$, $\chi|_{[2,\infty)}=0$ and $|\chi'|\leq 2$. Then we have 
\begin{equation}\label{eq:B-Integral}
\frac{i}2\int_{B_x} f_y\wedge \overline{f_y}\le  \frac{i}2\int_M \chi(\rho_x/\widehat{r}_x) f_y\wedge \overline{f_y},
\end{equation}   
where $\rho_x:=d(\cdot,x)$. The well-known property of the Bergman projection yields
\[
\frac{i}2\int_M \chi (\rho_x/\widehat{r}_x) f_y\wedge \overline{K_M(\cdot,a)}  =  \chi (\rho_x(a)/\widehat{r}_x) f_y(a) -u_0(a),\ \ \ \forall\,a\in{M},
\]
where $u_0$ is the $L^2$ minimal solution of the equation
\[
\bar{\partial}u=v:=\bar{\partial}(\chi (\rho_x/\widehat{r}_x) f_y).
\]
In particular,
\begin{equation}\label{eq:B-Projection}
\frac{i}{2} \int_{M} \chi (\rho_x/\widehat{r}_x) f_y\wedge \overline{f_y}=-u^\ast_0(y),
\end{equation}
for $\chi (\rho_x/\widehat{r}_x)|_{2B_y}=0$. Fix $\tau<\lambda_1(M)/9$. Put
\[
\varphi=-2\sqrt{\tau} \rho.
\]
Clearly, $\varphi$ is a  Lipschitz continuous function on $M$ which satisfies
\[
|\partial \varphi|^2=|\nabla \varphi|^2/4\le \tau\ \ \ {\rm a.e.}
\]
By virtue of Proposition \ref{prop:DF},  we have
\begin{eqnarray*}
\int_M |u_0|^2e^{\varphi}dV
&\leq& {\rm const}_\tau  \int_M |v|^2e^{\varphi}dV\\
&\leq& {\rm const}_\tau \widehat{r}_x^{\,-2} \int_{2B_x\backslash B_x} |f_y|^2e^{-2\sqrt{\tau} \rho}dV\\
&\leq& {\rm const}_\tau\,\widehat{r}_x^{\,-2}e^{-2\sqrt{\tau} \rho(x)}\int_M |f_y|^2dV\\
&\leq& {\rm const}_\tau\,\widehat{r}_x^{\,-2}e^{-2\sqrt{\tau}\rho(x)}K_M^\ast(y,y).
\end{eqnarray*}
Since $u_0$ is holomorphic in $B_y$, it follows  that 
\begin{eqnarray*}
|u_0(y)|^2
&\leq&  |K_{B_y}(y,y)|\int_{B_y}|u_0|^2\\
&\leq& {\rm const}_\tau e^{2\sqrt{\tau} \rho(y)} |K_{B_y}(y,y)| \int_{B_y}|u_0|^2 e^{\varphi}\\
&\leq& {\rm const}_\tau |K_{B_y}(y,y)| K_M^\ast(y,y)\,\widehat{r}_x^{\,-2}e^{2\sqrt{\tau}[\rho(y)-\rho(x)]}.
\end{eqnarray*}
In other words,
\begin{equation}\label{eq:PointwiseBound}
|u^\ast_0(y)|\leq {\rm const}_\tau K_{B_y}^\ast(y,y)^{1/2} K_M^\ast(y,y)^{1/2} \widehat{r}_x^{\,-1}e^{\sqrt{\tau} [\rho(y)-\rho(x)]}.
\end{equation}
This combined with (\ref{eq:B-Integral}) and (\ref{eq:B-Projection}) yields the conclusion.
\end{proof}

By Lemma \ref{lm:upper_estimate_B} and Lemma \ref{lm:Bergman_disc},  we obtain
\begin{equation}\label{eq:upper_B_1}
|\mathcal{B}_M(x,y)|\leq C_\tau e^{4\eta(x)}\frac{K_{B_y}^\ast(y,y)^{1/2}}{K_M^\ast(y,y)^{1/2}}\widehat{r}_x^{\,-1}e^{\sqrt{\tau}(\rho(y)-\rho(x))}.
\end{equation}

The main result of this section is the following

\begin{proposition}\label{prop:Key}
If $B_x\cap{B_y}=\varnothing$, then for every $0<\tau<\lambda_1(M)/9$ there exists a constant $C>0$ such that
\begin{equation}\label{eq:upper_B}
|\mathcal{B}_M(x,y)|\leq C \widehat{r}_x^{\,-1} e^{4\eta(x)+2\eta(y)}e^{\sqrt{\tau}(\rho(y)-\rho(x))}.
\end{equation}
\end{proposition}

By virtue of  \eqref{eq:upper_B_1},  it suffices to verify the following

\begin{lemma}\label{lm:Bergman_disc_compare}
There exists a numerical constant $C_2>0$ such that 
$$
|K_M(y,y)|\geq C_2^{-1}e^{-4\eta(y)}|K_{B_y}(y,y)|.
$$
\end{lemma}

\begin{proof}
Take $\tilde{f}_y\in \mathcal H(B_y)$ such that $|\tilde{f}_y(y)|^2=|K_{B_y}(y)|$ and $\|\tilde{f}_y\|=1$. Let $\kappa$ be the same cut-off function as in Lemma \ref{lm:upper_estimate_B}. Then a similar application of the Donnelly-Fefferman estimate yields a solution of the equation
\[
\bar{\partial} u=\tilde{f}_y \bar{\partial}\kappa(-\log(-g_M(\cdot,y))+\log \eta(y)),
\]
which satisfies
\begin{eqnarray*}
\int_M |u|^2 e^{-2 g_M(\cdot,y)} dV
&\leq&  C_3 e^{4\eta(y)} \int_{B_y}|\tilde{f}_y|^2dV=C_3 e^{4 \eta(y)}
\end{eqnarray*}
for suitable numerical constant $C_3>0$.  
Set
\[
\widetilde{F}:=\tilde{f}_y \kappa(-\log(-g_M(\cdot,y))+\log \eta(y))-u.
\]
Clearly, we have $\widetilde{F}\in \mathcal H$, $\widetilde{F}(y)=\tilde{f}_y(y)$ and
\begin{eqnarray*}
\int_M |\widetilde{F}|^2 dV
& \leq & 2\int_{B_y} |\tilde{f}_y|^2 dV+2\int_M |u|^2 dV\\
& \leq& 2+2C_3 e^{4\eta(y)}.
\end{eqnarray*}
Thus 
\[
|K_M(y,y)|\ge \frac{|\widetilde{F}(x)|^2}{\|\widetilde{F}\|^2}\ge \left(2+2C_3 e^{4\eta(y)}\right)^{-1} |K_{B_y}(y,y)|.\qedhere
\]
\end{proof}

\section{Proof of Theorem \ref{th:Main}}
Let $\varpi:\mathbb{D}\rightarrow M$ be the universal covering mapping and $\widetilde{x}\in\varpi^{-1}(x)$.  Recall that $\varpi(B_{\widehat{r}_x}(\widetilde{x}))=B_{\widehat{r}_x}(x)=B_x$. Thus
\[
|B_x|=|B_{\widehat{r}_x}(\widetilde{x})|=|B_{\widehat{r}_x}(0)|=4\pi\sinh^2(\widehat{r}_x/2)
\geq \pi\widehat{r}_x^{\,2}.
\]
Suppose that \eqref{eq:injectivity_radius} holds with $c_0>\sqrt{12C_0/\pi}$. It follows that $\widehat{r}_x\geq c_0\lambda_1(M)^{-3/4}\rho(x)^{-1/2}$ holds for $\rho(x)\ge R\gg 1$. Moreover,
\[
\eta(x)=C_0\lambda_1(M)^{-1}|B_x|^{-1}\leq C_0\pi^{-1}\lambda_1(M)^{-1}\widehat{r}_x^{\,-2}<\frac{1}{12}\lambda_1(M)^{1/2}\rho(x).
\]
Thus we may choose $\tau<\lambda_1(M)/9$ such that 
\[
\sqrt{\tau} \rho(x) - 4\eta(x)-\log{\widehat{r}_x^{\,-1}} \geq \varepsilon \rho(x)
\]
for suitable constant $\varepsilon>0$. Since $\sqrt{\tau}\rho(y)+ 2\eta(y) \leq \beta \rho(y)$ for some constant $\beta>0$, it follows from (\ref{eq:upper_B}) that
\[
\mathcal B_M(x,y)\lesssim e^{\beta\rho(y)-\varepsilon \rho(x)}\le 1/2
\]
whenever $\rho(y)\ge R=R(\varepsilon,\beta)\gg 1$ and $\rho(y)< \frac{\varepsilon}{2\beta}\cdot \rho(x)$. Thus
\[
d_B(x,y) \ge \sqrt{1-\mathcal B_M(x,y)}\ge \frac{\sqrt{2}}2.
\]
Now fix $x\in M$ with $\rho(x)\gg 1$. Let $c$ be a piece-wise smooth curve which joints $x_0$ to $x$. We may choose a finite number of points $\{x_k\}_{k=1}^n\subset c$ with the following order
\[
x_0\rightarrow x_1\rightarrow x_2\rightarrow \cdots \rightarrow x_n,
\]
such that 
\[
\rho(x_k)=\frac{\varepsilon}{2\beta} \cdot \rho(x_{k+1})\ \ \ \text{and}\ \ \ \rho(x)\le \frac{2\beta}{\varepsilon} \cdot \rho(x_n).
\]
It is easy to see that 
\[
n\asymp \log \rho(x_n)\gtrsim \log [1+\rho(x)]
\]
where the implicit constants are independent of the choice of $c$. It follows that the Bergman length $|c|_B$ of $c$ satisfies
\[
|c|_B \ge \sum_{k=1}^{n-1} d_B(x_k,x_{k+1}) \gtrsim n  \gtrsim \log [1+\rho(x)],
\]
from which the assertion immediately follows.

\end{document}